\documentclass[article]{elsarticle}

\usepackage{lineno,hyperref}
\modulolinenumbers[5]

\usepackage[english]{babel}
\usepackage{amsmath,amsfonts,amssymb,amsthm}
\usepackage{graphicx}
\usepackage[dvipsnames]{xcolor}
\newtheorem{theorem}{Theorem}[section]
\newtheorem{lemma}[theorem]{Lemma}

\newcommand{\dd}{\mathop{}\!\mathrm{d}}

\journal{journal}

\begin{document}

\begin{frontmatter}

\title{Bivariate collocation for computing $R_{0}$ in epidemic models with two structures}


\author[address1]{Dimitri Breda\corref{mycorrespondingauthor}}
\ead{dimitri.breda@uniud.it}

\author[address1]{Simone De Reggi}
\ead{dereggi.simone@spes.uniud.it}

\author[address1,address2,address3]{Francesca Scarabel}
\ead{francesca.scarabel@manchester.ac.uk}

\author[address1]{Rossana Vermiglio}
\ead{rossana.vermiglio@uniud.it}

\author[address2]{Jianhong Wu}
\ead{wujh@yorku.ca}

\address[address1]{CDLab -- Computational Dynamics Laboratory\\Department of Mathematics, Computer Science and Physics -- University of Udine\\Via delle Scienze 206, 33100, Italy}

\address[address2]{LIAM -- Laboratory for Industrial and Applied Mathematics\\Department of Mathematics and Statistics -- York University\\4700 Keele Street, Toronto, ON M3J 1P3, Canada}

\address[address3]{Department of Mathematics, The University of Manchester, Oxford Rd, M13 9PL, Manchester (UK), Joint UNIversities Pandemic and Epidemiological Research, \url{https://maths.org/juniper/}}

\cortext[mycorrespondingauthor]{Corresponding author}

\begin{abstract}
Structured epidemic models can be formulated as first-order hyperbolic PDEs, where the ``spatial'' variables represent individual traits, called {\it structures}. For models with two structures, we propose a numerical technique to approximate $R_{0}$, which measures the transmissibility of an infectious disease and, rigorously, is defined as the dominant eigenvalue of a next-generation operator.
Via bivariate collocation and cubature on tensor grids, the latter is approximated with a finite-dimensional matrix, so that its dominant eigenvalue can easily be computed with standard techniques.
We use test examples to investigate experimentally the behavior of the approximation: the convergence order appears to be infinite when the corresponding eigenfunction is smooth, and finite for less regular eigenfunctions.
To demonstrate the effectiveness of the technique for more realistic applications, we present a new epidemic model structured by demographic age and immunity, and study the approximation of $R_{0}$ in some particular cases of interest.
\end{abstract}


\end{frontmatter}


\section{Introduction}
\label{s_{i}ntroduction}
In the mathematical modeling of epidemics, the {\it basic reproduction number} $R_{0}$ measures the average number of secondary cases produced by a typical infected individual in a fully susceptible population. As such, it is widely used to evaluate the potential of the spread of an infectious disease outbreak, the speed of the spread of the disease if an outbreak does occur and its controllability through public health interventions. Its estimation is thus a primary objective and, in this respect, the case of coronavirus disease 2019 (COVID-19) is a recent example of the essential role played by this key quantity \cite{bmk20,kun20,ps20,tang20}.

On the one hand, from the dynamical systems point of view, the threshold value $R_{0}=1$ represents a transcritical bifurcation of the disease-free equilibrium (absence of infected) from being locally asymptotically stable to unstable, with the emergence of an equilibrium in which part of the population has been infected. The standard analysis of the dynamic behaviors is based on linearizing the nonlinear model at hands around the disease-free equilibrium to investigate the local stability of the latter, typically through an eigenvalue-based approach. 

On the other hand, in an attempt to capture a more realistic portrait of the infection transmission, some modern population models account for individual variability by introducing continuous structuring variables, or, briefly, {\it structures}. These represent physical or physiological traits that determine the epidemiological or vital properties of the individuals \cite{mr08,md86}. The mathematical description is mostly based on Partial Differential Equations (PDEs) in time and a second variable, that we call ``space'' in short, that collects all the (continuous) structuring variables. Notable instances of structures are chronological age, time since infection, some immunity level or even proper spatial position and distancing. 

The combination of both aforementioned aspects results in the study of a linear evolution equation on an abstract space of functions depending on each structuring variable. This evolution can be suitably described by means of two operators, say $B$ and $M$, where the former describes the ``birth'' process of newly infected individuals, whereas the latter collects all the other possible processes including, e.g., death, recovery, evolution of the individual trait, or transfer to other classes (for instance quarantine). Then, under mild and classical hypotheses, one can define the so-called {\it Next Generation Operator (NGO)} -- which maps a given composition of the population to the composition after one generation of infection -- as $BM^{-1}$ and characterize $R_{0}$ as its spectral radius \cite{bcr17b}.

The latter characterization can be favorably exploited for numerical purposes. Indeed, any reasonable finite-dimensional approximation of the NGO provides a matrix whose eigenvalues allow in principle to estimate $R_{0}$. This general idea has been first proposed in \cite{bfrv20} in terms of discretizing separately $B$ and $M$. Therein, an approach based on pseudospectral collocation for populations with a single structure has been thoroughly investigated by experimenting on several models from either epidemiology or ecology. A proper rigorous analysis of the convergence of this methodology has been carried out in \cite{bkrv20}, confirming the expected spectral accuracy \cite{tref00} in the presence of smooth coefficients, as well as promising behaviors even in the absence of full regularity or of compactness of the NGO. The outcome is a quite reliable tool, with convergence potentially of infinite order\footnote{By ``convergence of infinite order'' we mean that the error decays faster than any finite polynomial order (or even exponentially): this is the case in general of (pseudo)spectral methods when applied to functions $C^{\infty}$ (or analytic), see, e.g., \cite{tref00}.}, thus more efficient than the only other two methods available in the literature \cite{guo19,kun17}, based respectively on $\theta$- and Euler discretization schemes. Rather than for the accuracy itself, the advantage is that of working with much smaller matrices, reducing the computational burden and thus allowing for stability and bifurcation analyses, which are typically performed in a continuation framework \cite{doe07}. This is particularly valuable in the presence of varying or uncertain model parameters, as typically happens in realistic contexts.

The current work extends from the numerical and experimental points of view the pseudospectral approach of \cite{bfrv20,bkrv20} to population models with two structures, by resorting to bivariate collocation and cubature on tensor grids. 

Models including more than one structure have long been suggested in the literature, see, e.g., \cite{web08} and the reference therein, as they allow to better describe real phenomena. Nevertheless, it is only very recently that a systematic mathematical investigation has gained renovated interest \cite{khr20b,khr20a}, providing also the theoretical background essential to numerical developments. To the best of the authors' knowledge, the resulting technique is the first available numerical method for computing $R_{0}$ in the case of two structures, thus providing modelers with a tool to study applications in more complex and realistic settings. Moreover, it is designed to tackle a quite versatile class of models, which includes instances from both \cite{web08} and \cite{khr20b,khr20a}.

The extension to the case of two structures is not completely straightforward from the numerical point of view due to, e.g., the necessity of resorting to cubature and vectorization for the sake of implementation. Moreover, we anticipate that a straightforward extension to the case of two structures of the convergence analysis developed in \cite{bkrv20} is not immediate (and hence out of the scope of the present work), as the search for a characteristic equation relies on tools for PDEs rather than on more standard tools for ordinary differential equations as those used in \cite{bkrv20}.

The contents are organized as follows. In Section \ref{s_linear} we first resume from \cite{bcr17b} the main aspects leading to the definition of $R_{0}$ and then introduce the class of models of interest, i.e., first-order hyperbolic PDEs with two ``spatial'' variables and nonlocal boundary conditions of integral type. The numerical approach is presented in Section \ref{s_numerical}. A thorough experimental study of convergence is performed in Section \ref{s_implementation}. In Section \ref{s_ageimmunity} we apply the method to an age-immunity model we introduced for studying childhood diseases and the relevant immunization programs. Some concluding remarks are given in Section \ref{s_concluding}. Matlab demos are freely available at \url{http://cdlab.uniud.it/software}.
\section{Linear evolution, $R_{0}$ and models}
\label{s_linear}
Let $X$ be a Banach lattice of real functions representing the density of individuals of a population depending on some structuring variables. Following \cite{bcr17b} and what anticipated in the introduction, we consider a linear evolution equation of the form
\begin{equation}\label{linear}
v'(t)=Bv(t)-Mv(t),\quad t\geq0,
\end{equation}
where $B:X\to X$ is a linear operator representing the birth process of the population and $M:\mathcal{D}(M)\subseteq X\to X$ is a linear operator representing all the other processes. Typically, $\mathcal{D}(M)$ is a subspace characterized by some degree of smoothness of $M$ and additional linear constraints, $B$ is positive and bounded and $-M$ generates a strongly-continuous semigroup $\{T(t)\}_{t\geq0}$ of positive linear operators, with strictly negative spectral bound. The latter guarantees extinction in the absence of birth, as well as the invertibility of $M$ with $M^{-1}=\int_{0}^\infty T(t)\dd t$. Then the NGO $BM^{-1}:X\to X$ is well defined, positive and bounded, and one can characterize $R_{0}$ as its spectral radius. $R_{0}$ is actually a non-negative spectral value \cite{sha74} and if, in addition, $BM^{-1}$ is also compact with positive spectral radius, then the Krein-Rutman theorem \cite{kr48} ensures that $R_{0}$ is a positive eigenvalue, i.e., a solution $\lambda>0$ of 
\begin{equation}\label{eig}
BM^{-1}\psi=\lambda\psi
\end{equation}
for some positive eigenfunction $\psi$. Equivalently,  $\lambda$ satisfies
\begin{equation}\label{geig}
B\phi=\lambda M\phi
\end{equation}
with $\phi=M^{-1}\psi\in\mathcal{D}(M)$.

\bigskip
After linearization around the trivial steady state, several dynamical models of structured populations can be recast as \eqref{linear}, independently of them being based on ordinary, partial or delay differential equations, or even renewal equations. In this work we focus our attention on populations with two structures. As a reference we consider from \cite{khr20a} the initial-boundary value problem for a first-order linear hyperbolic PDE
\begin{equation}\label{ruan}
\left\{\setlength\arraycolsep{0.1em}\begin{array}{rcl}
\partial_{t} u(t,x,y)&+&\displaystyle \partial_{x}u(t,x,y)+\partial_{y}u(t,x,y)=\\[4mm]
&-&\displaystyle\mu(x,y)u(t,x,y)+\int_{x_{0}}^{\bar x}\int_{y_{0}}^{\bar y}K(x,y,\xi,\sigma)u(t,\xi,\sigma)\dd\sigma\dd\xi,\\[4mm]
u(t,x,y_{0})&=&\displaystyle\int_{x_{0}}^{\bar x}\int_{y_{0}}^{\bar y}\alpha(x,\xi,\sigma)u(t,\xi,\sigma)\dd\sigma\dd\xi,\\[4mm]
u(t,x_{0},y)&=&\displaystyle\int_{x_{0}}^{\bar x}\int_{y_{0}}^{\bar y}\beta(y,\xi,\sigma)u(t,\xi,\sigma)\dd\sigma\dd\xi,\\[4mm]
u(0,x,y)&=&\displaystyle u_{0}(x,y),
\end{array}\right.
\end{equation}
where $u(t,x,y)$ is the density of the given population at time $t\geq0$ depending on the two structuring variables $x\in[x_{0},\bar x]$ and $y\in[y_{0},\bar y]$ with $x_{0}<\bar x$ and $y_{0}<\bar y$ for given $x_{0},\bar x,y_{0},\bar y\in\mathbb{R}$. Above, $\mu$, $K$, $\alpha$, $\beta$ and $u_{0}$ are given functions, non-negative in their domain and satisfying the required conditions for existence and uniqueness, see \cite[Assumption 2.1]{khr20a}. We also assume the compatibility condition
\begin{equation}\label{cc}
\int_{x_{0}}^{\bar x}\int_{y_{0}}^{\bar y}\alpha(x_{0},\xi,\sigma)u(t,\xi,\sigma)\dd\sigma\dd\xi=\int_{x_{0}}^{\bar x}\int_{y_{0}}^{\bar y}\beta(y_{0},\xi,\sigma)u(t,\xi,\sigma)\dd\sigma\dd\xi
\end{equation}
to hold true.

\bigskip
Model \eqref{ruan} can be cast into \eqref{linear} by considering $X:=L^{1}([x_{0},\bar x]\times[y_{0},\bar y],\mathbb{R})$ and by defining $v(t):=u(t,\cdot,\cdot)$ for $t\geq0$, as well as
\begin{equation*}
(B\phi)(x,y):=\int_{x_{0}}^{\bar x}\int_{y_{0}}^{\bar y}K(x,y,\xi,\sigma)\phi(\xi,\sigma)\dd\sigma\dd\xi
\end{equation*}
and
\begin{equation*}
(M\phi)(x,y):=\partial_{x}\phi(x,y)+\partial_{y}\phi(x,y)+\mu(x,y)\phi(x,y)
\end{equation*}
with domain
\begin{equation*}
\setlength\arraycolsep{0.1em}\begin{array}{rcl}
\mathcal{D}(M):=\bigg\{\phi\in X\ &:\displaystyle&\ \partial_{x}\phi+\partial_{y}\phi\in X,\\[4mm]
&&\displaystyle\phi(x,y_{0})=\int_{x_{0}}^{\bar x}\int_{y_{0}}^{\bar y}\alpha(x,\xi,\sigma)\phi(\xi,\sigma)\dd\sigma\dd\xi\ \text{ for }x\in[x_{0},\bar x]\text{ and }\\[4mm]
&&\displaystyle\phi(x_{0},y)=\int_{x_{0}}^{\bar x}\int_{y_{0}}^{\bar y}\beta(y,\xi,\sigma)\phi(\xi,\sigma)\dd\sigma\dd\xi\ \text{ for }y\in[y_{0},\bar y]\bigg\}.
\end{array}
\end{equation*}
The resulting NGO $BM^{-1}$ is compact, see \cite[Lemma 5.4]{khr20a} (where the NGO is denoted by $F_{0}$). Moreover, $R_{0}=\rho(BM^{-1})$ is a simple positive eigenvalue \cite[Proposition 5.6]{khr20a} and the classical threshold stability theorem holds \cite[Theorem 5.11]{khr20a}\footnote{Actually \cite{khr20a} considers the case $K\equiv0$, but it is not difficult to argue that the same results still hold by assuming mild conditions on the kernel $K$ given the linearity of the additional term.}. 

\bigskip
The compactness of the NGO is an essential working hypothesis in view of the numerical treatment we propose in Section \ref{s_numerical}. Indeed, the general idea we wish to follow is that of reducing \eqref{geig} to a standard (generalized) eigenvalue problem for matrices through discretization. Then, the resulting dominant eigenvalue is a candidate to approximate $R_{0}$. Actually, in \cite{bfrv20,bkrv20} test cases of lack of compactness are reported for which the methodology is still able to provide accurate approximations.

On the other hand, the literature (see, e.g., \cite{web08}) offers also other models of populations with two structures that do not necessarily belong to the class \eqref{ruan}, but are still described by first-order hyperbolic PDEs and differ from \eqref{ruan} just for additional features that we believe are still amenable of the theoretical analysis concerning the semigroup approach developed in \cite{khr20a} at the price of an increased technicality. Although the latter analysis is out of the scope of the present work, by considering also the potential applicability of an approach based on discretization, in the sequel we take as a reference the following general problem:
\begin{equation}\label{model}
\left\{\setlength\arraycolsep{0.1em}\begin{array}{rcl}
\partial_{t} u(t,x,y)&+&\displaystyle a(x,y)\partial_{x}[b(x,y)u(t,x,y)]+c(x,y)\partial_{y}[d(x,y)u(t,x,y)]\\[4mm]
&=&\displaystyle-\mu(x,y)u(t,x,y)+\int_{x_{0}}^{\bar x}\int_{y_{0}}^{\bar y}K(x,y,\xi,\sigma)u(t,\xi,\sigma)\dd\sigma\dd\xi,\\[4mm]
u(t,x,y_{0})&=&\displaystyle\int_{x_{0}}^{\bar x}\int_{y_{0}}^{\bar y}\alpha(x,\xi,\sigma)u(t,\xi,\sigma)\dd\sigma\dd\xi,\\[4mm]
u(t,x_{0},y)&=&\displaystyle\int_{x_{0}}^{\bar x}\int_{y_{0}}^{\bar y}\beta(y,\xi,\sigma)u(t,\xi,\sigma)\dd\sigma\dd\xi,\\[4mm]
u(0,x,y)&=&\displaystyle u_{0}(x,y).
\end{array}\right.
\end{equation}
Note that \eqref{model} differs from \eqref{ruan} just for the presence of the coefficients $a$, $b$, $c$ and $d$, which are assumed to be non-negative functions\footnote{$b$ can be also non-positive if the relevant boundary condition is prescribed at $\bar x$ (and similarly for $d$).}. In any case, \eqref{model} can be recast as \eqref{linear} with the same choices of $X$ and $B$, but redefining $M$ as
\begin{equation}\label{M}
\setlength\arraycolsep{0.1em}\begin{array}{rcl}
(M\phi)(x,y)&:=&a(x,y)\partial_{x}[b(x,y)\phi(x,y)]+c(x,y)\partial_{y}[d(x,y)\phi(x,y)]\\[1mm]
&&+\mu(x,y)\phi(x,y)
\end{array}
\end{equation}
with domain
\begin{equation}\label{DM}
\setlength\arraycolsep{0.1em}\begin{array}{rcl}
\mathcal{D}(M):=\bigg\{&&\displaystyle\phi\in X\ :\ a\partial_{x}(b\phi)+c\partial_{y}(d\phi)\in X,\\[2mm]
&&\displaystyle\phi(x,y_{0})=\int_{x_{0}}^{\bar x}\int_{y_{0}}^{\bar y}\alpha(x,\xi,\sigma)\phi(\xi,\sigma)\dd\sigma\dd\xi,\ x\in[x_{0},\bar x],\text{ and }\\[4mm]
&&\displaystyle\phi(x_{0},y)=\int_{x_{0}}^{\bar x}\int_{y_{0}}^{\bar y}\beta(y,\xi,\sigma)\phi(\xi,\sigma)\dd\sigma\dd\xi,\ y\in[y_{0},\bar y]\bigg\}.
\end{array}
\end{equation}
\section{The numerical approach}
\label{s_numerical}
Following \cite{bfrv20,bkrv20}, we use collocation to discretize \eqref{model} in order to get a finite-dimensional version of \eqref{geig}. As we deal with models with two structures, we necessarily resort to bivariate collocation on $[x_{0},\bar x]\times[y_{0},\bar y]$. In particular, we adopt the standard approach based on tensor grids, leaving to Section \ref{s_concluding} comments relevant to other choices.

\bigskip
For $n$ and $m$ positive integers, let $x_{0}<x_{1}<\dots<x_{n}=\bar x$ be $n+1$ points in $[x_{0},\bar x]$, $y_{0}<y_{1}<\dots<y_{m}=\bar y$ be $m+1$ points in $[y_{0},\bar y]$ and $\Pi_{n,m}$ be the space of bivariate polynomials on $[x_{0},\bar x]\times[y_{0},\bar y]$ of degree at most $n$ in the first variable and at most $m$ in the second variable. Let, moreover, $X_{n,m}:=\mathbb{R}^{(n+1)(m+1)}$ be the discrete counterpart of $X$, meaning that an element $\phi\in X$ is thought as approximated by an element $\Phi\in X_{n,m}$ according to\footnote{We observe that although pointwise evaluation is meaningless in $L^{1}$, the elements $\phi$ that we are going to approximate are eigenfunctions, which are in general regular enough.}
\begin{equation*}
\phi(x_{i},y_{j})\approx\Phi_{i,j},\quad i=0,1,\ldots,n,\;j=0,1,\ldots,m,
\end{equation*}
where the components of $\Phi$ are ordered according to the sequence of double indexes $(0,0),\ldots,(0,m),(1,0),\ldots,(1,m),\ldots,(n,0),\ldots,(n,m)$\footnote{In place of $(i,j)$ one can use a single index $k=0,1,\ldots,(n+1)(m+1)$ with $k=i(m+1)+j$.}.

\bigskip
Now we construct in $X_{n,m}$ a finite-dimensional version
\begin{equation}\label{dgeig}
B_{n,m}\Phi=\lambda M_{n,m}\Phi,
\end{equation}
of \eqref{geig} as follows. Consider $\phi_{n,m}\in\Pi_{n,m}$ collocating \eqref{model} as
\begin{equation}\label{dgeigp}
\begin{cases}
(B\phi_{n,m})(x_{i},y_{j})=\lambda (M\phi_{n,m})(x_{i},y_{j}),&i=1,\ldots,n,\;j=1,\ldots,m,\\[2mm]
\displaystyle \phi_{n,m}(x_{i},y_{0})=\int_{x_{0}}^{\bar x}\int_{y_{0}}^{\bar y}\alpha(x_{i},\xi,\sigma)\phi_{n,m}(\xi,\sigma)\dd\sigma\dd\xi,&i=1,\ldots,n\\[4mm]
\displaystyle \phi_{n,m}(x_{0},y_{j})=\int_{x_{0}}^{\bar x}\int_{y_{0}}^{\bar y}\beta(y_{j},\xi,\sigma)\phi_{n,m}(\xi,\sigma)\dd\sigma\dd\xi,&j=0,1,\ldots,m,
\end{cases}
\end{equation}
which suitably takes into account the boundary conditions characterizing $\mathcal{D}(M)$ in \eqref{DM}. Note that, given \eqref{cc}, either one of the two conditions can be imposed at the node $(x_{0},y_{0})$, and we choose the second one without loss of generality. Then \eqref{dgeig} is recovered by setting
\begin{equation*}
\Phi_{i,j}:=\phi_{n,m}(x_{i},y_{j}),\quad i=0,1,\ldots,n,\;j=0,1,\ldots,m,
\end{equation*}
and by introducing the finite-dimensional counterparts $B_{n,m},M_{n,m}:X_{n,m}\to X_{n,m}$ of $B$ and $M$ respectively as
\begin{equation}\label{Bnm}
\begin{cases}
[B_{n,m}\Phi]_{i,j}:=(B\phi_{n,m})(x_{i},y_{j}),&i=1,\ldots,n,\;j=1,\ldots,m,\\[2mm]
[B_{n,m}\Phi]_{i,j}=0,&\text{otherwise},
\end{cases}
\end{equation}
and
\begin{equation}\label{Mnm}
\left\{\setlength\arraycolsep{0.1em}\begin{array}{rcll}
[M_{n,m}\Phi]_{i,j}&:=&(M\phi_{n,m})(x_{i},y_{j}),&\; i=1,\ldots,n,\;j=1,\ldots,m,\\[2mm]
\displaystyle [M_{n,m}\Phi]_{0,j}&:=&\phi_{n,m}(x_{0},y_{j})&\\[2mm]
&&\displaystyle-\int_{x_{0}}^{\bar x}\int_{y_{0}}^{\bar y}\beta(y_{j},\xi,\sigma)\phi_{n,m}(\xi,\sigma)\dd\sigma\dd\xi,&\; j=1,\ldots,m,\\[4mm]
\displaystyle [M_{n,m}\Phi]_{i,0}&:=&\phi_{n,m}(x_{i},y_{0})&\\[2mm]
&&\displaystyle-\int_{x_{0}}^{\bar x}\int_{y_{0}}^{\bar y}\alpha(x_{i},\xi,\sigma)\phi_{n,m}(\xi,\sigma)\dd\sigma\dd\xi,&\; i=1,\ldots,n\\[4mm]
\displaystyle [M_{n,m}\Phi]_{0,0}&:=&\phi_{n,m}(x_{0},y_{0})&\\[2mm]
&&\displaystyle-\int_{x_{0}}^{\bar x}\int_{y_{0}}^{\bar y}\beta(y_{0},\xi,\sigma)\phi_{n,m}(\xi,\sigma)\dd\sigma\dd\xi.
\end{array}\right.
\end{equation}
Note that the boundary conditions are implemented as zero conditions in $M_{n,m}$ and by annihilating the corresponding rows of $B_{n,m}$.

Once we have the matrix representations of the operators $B_{n,m}$ and $M_{n,m}$, an approximation $R_{0,n,m}$ of $R_{0}$ is computed as\footnote{Alternatively, one can solve the standard eigenvalue problem $B_{n,m}M_{n,m}^{-1}\Psi=\lambda\Psi$ as a discretization of \eqref{eig}. The analysis in \cite{bfrv20} shows that there is no particular advantage in doing so.}
\begin{equation*}
R_{0,n,m}:=\max{\left\{|\lambda|\in\mathbb{C}\ :\ \exists\ \Phi\in X_{n,m}\text{ such that }B_{n,m}\Phi=\lambda M_{n,m}\Phi\right\}}.
\end{equation*}
Of course, to get these matrices one has to consider to compute the action of both $B$ and $M$ on $\phi_{n,m}$ which, together with the boundary conditions, also requires to compute, e.g., double integrals. We leave a short discussion on these implementation aspects to Section \ref{s_implementation}.
\section{Implementation and testing}
\label{s_implementation}
We first discuss in Section \ref{s_discretization} some implementation choices resuming from Section \ref{s_numerical} and then perform a series of experiments in Section \ref{s_tests} to investigate experimentally the convergence properties of the proposed method.
\subsection{Discretization matrices}
\label{s_discretization}
Let us write $\phi_{n,m}$ in \eqref{dgeigp} by using the bivariate Lagrange representation
\begin{equation*}
\phi_{n,m}(x,y)=\sum_{i=0}^{n}\sum_{j=0}^{m}\ell_{x,i}(x)\ell_{y,j}(y)\Phi_{i,j},
\end{equation*}
where $\{\ell_{x,0},\ell_{x,1},\ldots,\ell_{x,n}\}$ is the Lagrange basis relevant to the collocation points in $[x_{0},\bar x]$ and $\{\ell_{y,0},\ell_{y,1},\ldots,\ell_{y,m}\}$ is the Lagrange basis relevant to the collocation points in $[y_{0},\bar y]$. Starting from \eqref{Bnm}, we get
\begin{equation*}
\setlength\arraycolsep{0.1em}\begin{array}{rcl}
(B\phi_{n,m})(x_{i},y_{j})&=&\displaystyle\int_{x_{0}}^{\bar x}\int_{y_{0}}^{\bar y}K(x_{i},y_{j},\xi,\sigma)\phi_{n,m}(\xi,\sigma)\dd\sigma\dd\xi\\[4mm]
&=&\displaystyle\sum_{k=0}^{n}\sum_{h=0}^{m}\Phi_{k,h}\int_{x_{0}}^{\bar x}\int_{y_{0}}^{\bar y}K(x_{i},y_{j},\xi,\sigma)\ell_{x,k}(\xi)\ell_{y,h}(\sigma)\dd\sigma\dd\xi.
\end{array}
\end{equation*}
In general, the double integral at the right-hand side above is not computable exactly. Therefore, for the sake of implementation, we resort to a cubature formula. In view of efficiency, it is convenient to adopt the formula based on the same points used for collocation. Let $w_{x,i}$, $i=0,1,\ldots,n$ and $w_{y,j}$, $j=0,1,\ldots,m$, be the quadrature weights associated to the collocation points in $[x_{0},\bar x]$ and $[y_{0},\bar y]$, respectively. Then we replace $[B_{n,m}\Phi]_{i,j}$ in \eqref{Bnm} for $i=1,\ldots,n$ and $j=1,\ldots,m$ with
\begin{equation*}
[\tilde B_{n,m}\Phi]_{i,j}:=\sum_{k=0}^{n}\sum_{h=0}^{m}w_{x,k}w_{y,h}K(x_{i},y_{j},x_{k},y_{h})\Phi_{k,h}.
\end{equation*}

As far as $M_{n,m}$ in \eqref{Mnm} is concerned, the same cubature above leads to similar approximations $[\tilde M_{n,m}\Phi]_{i,j}$ of $[M_{n,m}\Phi]_{i,j}$ for the choices of $i$ and $j$ involving the boundary conditions. In the remaining cases, i.e., $i=1,\ldots,n$ and $j=1,\ldots,m$, the action of $M$ on $\phi_{n,m}$ requires the use of the differentiation matrices relevant to the collocation points. As an illustration, we consider just the term involving the partial derivative with respect to the first variable. Then, starting from \eqref{M} and resorting to the bivariate interpolant of $b\phi_{n,m}$ we write
\begin{equation*}
\setlength\arraycolsep{0.1em}\begin{array}{rcl}
\left.\partial_{x}[b(x,y)\phi_{n,m}(x,y)]\right|_{(x,y)=(x_{i},y_{j})}&\approx&\displaystyle\sum_{k=0}^{n}\sum_{h=0}^{m}\left.\partial_{x}[\ell_{x,k}(x)\ell_{y,h}(y)]\right|_{(x,y)=(x_{i},y_{j})}b(x_{k},y_{h})\Phi_{k,h}\\[2mm]
&=&\displaystyle\sum_{k=0}^{n}\ell_{x,k}'(x_{i})b(x_{k},y_{j})\Phi_{k,j},
\end{array}
\end{equation*}
where we used the fact that $\ell_{y,h}(y_{j})=\delta_{h,j}$. Eventually, we replace $[M_{n,m}\Phi]_{i,j}$ for $i=1,\ldots,n$ and $j=1,\ldots,m$ with
\begin{equation*}
\setlength\arraycolsep{0.1em}\begin{array}{rcl}
[\tilde M_{n,m}\Phi]_{i,j} &:=&\displaystyle a(x_{i},y_{j})\sum_{k=0}^{n}\ell_{x,k}'(x_{i})b(x_{k},y_{j})\Phi_{k,j}\\[2mm]
&&+\displaystyle c(x_{i},y_{j})\sum_{h=0}^{m}\ell_{y,h}'(y_{j})d(x_{i},y_{h})\Phi_{i,h}+\mu(x_{i},y_{j})\Phi_{i,j}.
\end{array}
\end{equation*}

\bigskip
As far as the collocation points are concerned, we make the choice of Chebyshev extremal nodes in both variables. Indeed, they include $x_{0}$ and $y_{0}$ and thus allow for an easy handling of the boundary conditions as described in Section \ref{s_numerical}. Yet more importantly, they are shown to be near-optimal points also in the bivariate case \cite{cmv05}, with relevant Lebesgue constant growing as $O(\log n \cdot\log m)$ in uniform norm. Finally, the relevant univariate differentiation matrices can be computed explicitly \cite{tref00} and the associated quadrature, known as Clenshaw-Curtis formula, is spectrally accurate \cite{tref08}\footnote{The mentioned results refers to uniform norm. In our context $\|f\|_{X}\leq(\bar x-x_{0})(\bar y-y_{0})\|f\|_{\infty}$ trivially holds for sufficiently smooth $f$ (as eigenfunctions are in general) and $X=L^{1}([x_{0},\bar x]\times[y_{0},\bar y],\mathbb{R})$.}.

\bigskip
As a final remark, concerning the convergence of the spectrum and, in turn, of $R_{0}$, we expect a spectrally accurate behavior, i.e., an order of convergence proportional to the degree of smoothness of the concerned eigenfunctions and, therefore, of the model coefficients. Although a detailed analysis with rigorous proofs is  out of the scope of this computational and experimental work, to support our expectation we mention \cite[Theorem 1]{bbl02}, from which 
\begin{equation*}
\left\|f-p_{N}^{\ast}\right\|_{\infty}\leq\frac{C}{N^{\kappa}}\cdot\omega_{f,\kappa}\left(\frac{1}{N}\right),
\end{equation*}
holds for the best uniform approximation $p_{N}^{\ast}$ of $f\in C^{\kappa}$ with $N:=\min\{n,m\}$ and $C$ a positive constant, where
\begin{equation*}
\omega_{f,\kappa}(\delta):=\sup_{|\gamma|=\kappa}\left(\sup_{\|\nu-\eta\|\leq\delta}\left\|D^{\gamma}f(\nu)-D^{\gamma}f(\eta)\right\|\right)
\end{equation*}
for $\delta>0$ and $\|\cdot\|$ any norm on $\mathbb{R}^{2}$. Then estimates of both the interpolation and cubature errors follow similarly as for the univariate case, where also the Lebesgue constant of the collocation points appears.
\subsection{Numerical tests}
\label{s_tests}
In order to test the convergence properties of the proposed method, we use three benchmark instances of \eqref{model} for which both $R_{0}$ and the associated eigenfunction $\phi$ can be computed exactly, see Table \ref{t1} for the relevant data and coefficients. The latter are selected in such a way to obtain eigenfunctions $\phi$ of different regularity. Moreover, all the resulting eigenfunctions can be written in the form $\phi(x,y)=f(x)g(y)$. As a consequence, also the collocation polynomial can be written as the product $\phi_{n,m}(x,y)=p_{n}(x)q_{m}(y)$, leading to express the bivariate error as a function of the univariate errors, viz.
\begin{equation}\label{unibi}
\phi-\phi_{n,m}=e_{n}(f)g+fe_{m}(g)+e_{n}(f)e_{m}(g)
\end{equation}
for $e_{n}(f):=f-p_{n}$ and $e_{m}(g):=g-q_{m}$. It follows that the bivariate order is determined by the least univariate order.

\begin{table}
\begin{center}
\begin{tabular}{|c|c|c|c|}
\hline
&Example 1&Example 2&Example 3\\
\hline
$x_{0}$&$0$&$0$&$0$\\[1mm]
$\bar x$&$1$&$1$&$1$\\[1mm]
$y_{0}$&$\pi/6$&$0$&$0$\\[1mm]
$\bar y$&$\pi/4$&$1$&$2$\\[1mm]
$a(x,y)$&$(\cos y)/3$&$2x/15$&$1$\\[1mm]
$b(x,y)$&$1$&$1$&$1$\\[1mm]
$c(x,y)$&$(\sin y)/3$&$y/8$&$2y/7$\\[1mm]
$d(x,y)$&$1$&$1$&$1$\\[1mm]
$\mu(x,y)$&$(\cos y)/3$&$1/3$&$1$\\[1mm]
$K(x,y,\xi,\sigma)$&$e^{x}\cos y\sin y$&$x^{5/2}y^{8/3}$&$e^{-x}y^{7/2}$\\[1mm]
$\alpha(x,\xi,\sigma)$&$Ce^{x}/2$&$0$&$0$\\[1mm]
$\beta(y,\xi,\sigma)$&$C\sin y$&$0$&$Cy^{7/2}$\\[1mm]
$C$&$\frac{2}{(e-1)(\sqrt3-\sqrt2)}$&$-$&$\frac{9e}{2^{11/2}(e-1)}$\\[1mm]
\hline
\end{tabular}
\end{center}
\caption{Data and coefficients of model \eqref{model} for the numerical tests of Section \ref{s_tests}.}
\label{t1}
\end{table}

\paragraph{Example 1}For the choices listed in the second column of Table \ref{t1},  \eqref{geig} gives
\begin{equation*}
R_{0}=\frac{1}{C}\approx0.273066981413697,\qquad\phi(x,y)=e^{x}\sin y.
\end{equation*}
The trend of the errors for increasing $n=m$ on both $R_{0}$ and $\phi$ are reported in Figure \ref{f_ex12} (left). Convergence of infinite order is observed, which is reasonably expected being $\phi$ analytic.
\begin{figure}
\begin{center}
\includegraphics[width=.49\textwidth]{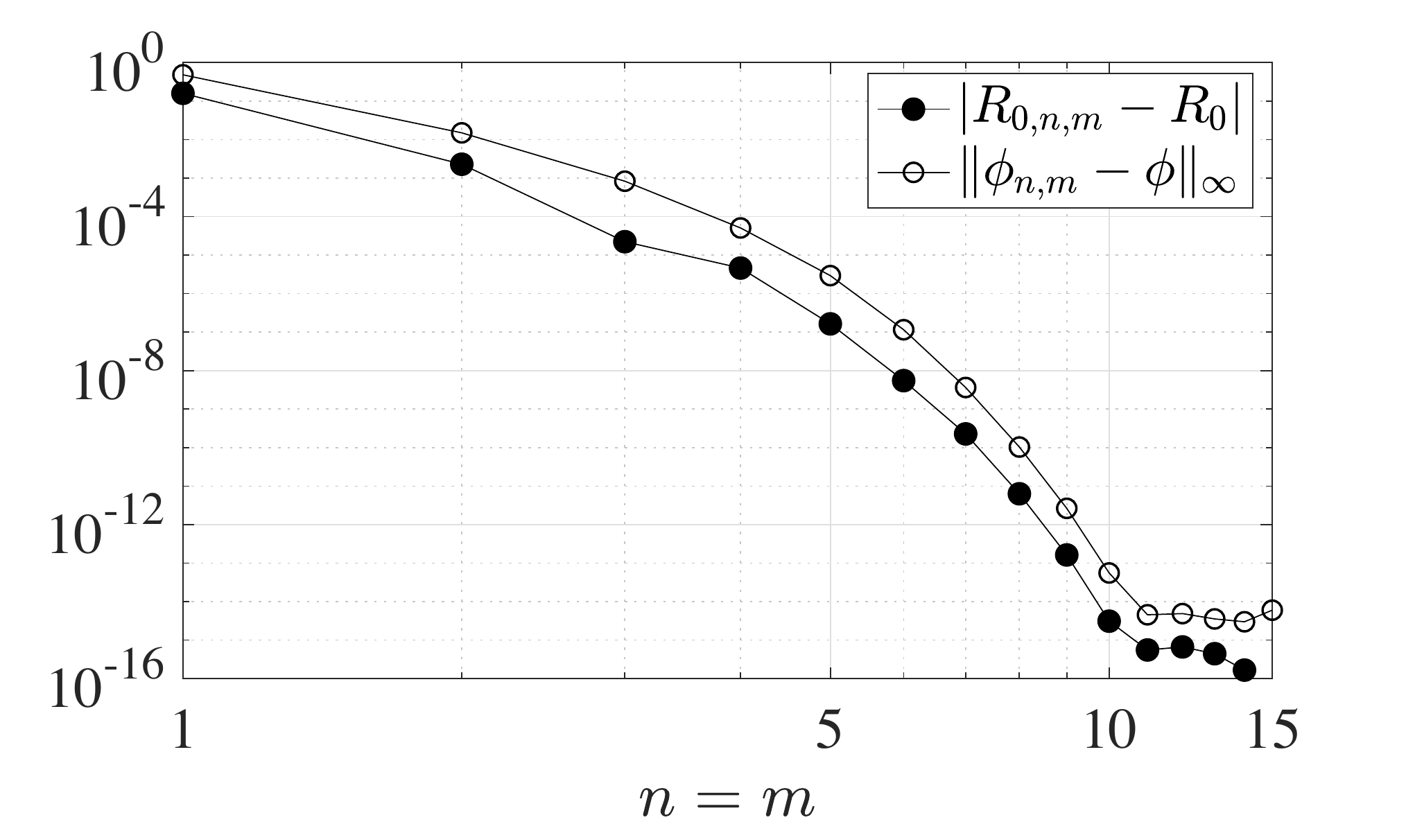}
\includegraphics[width=.49\textwidth]{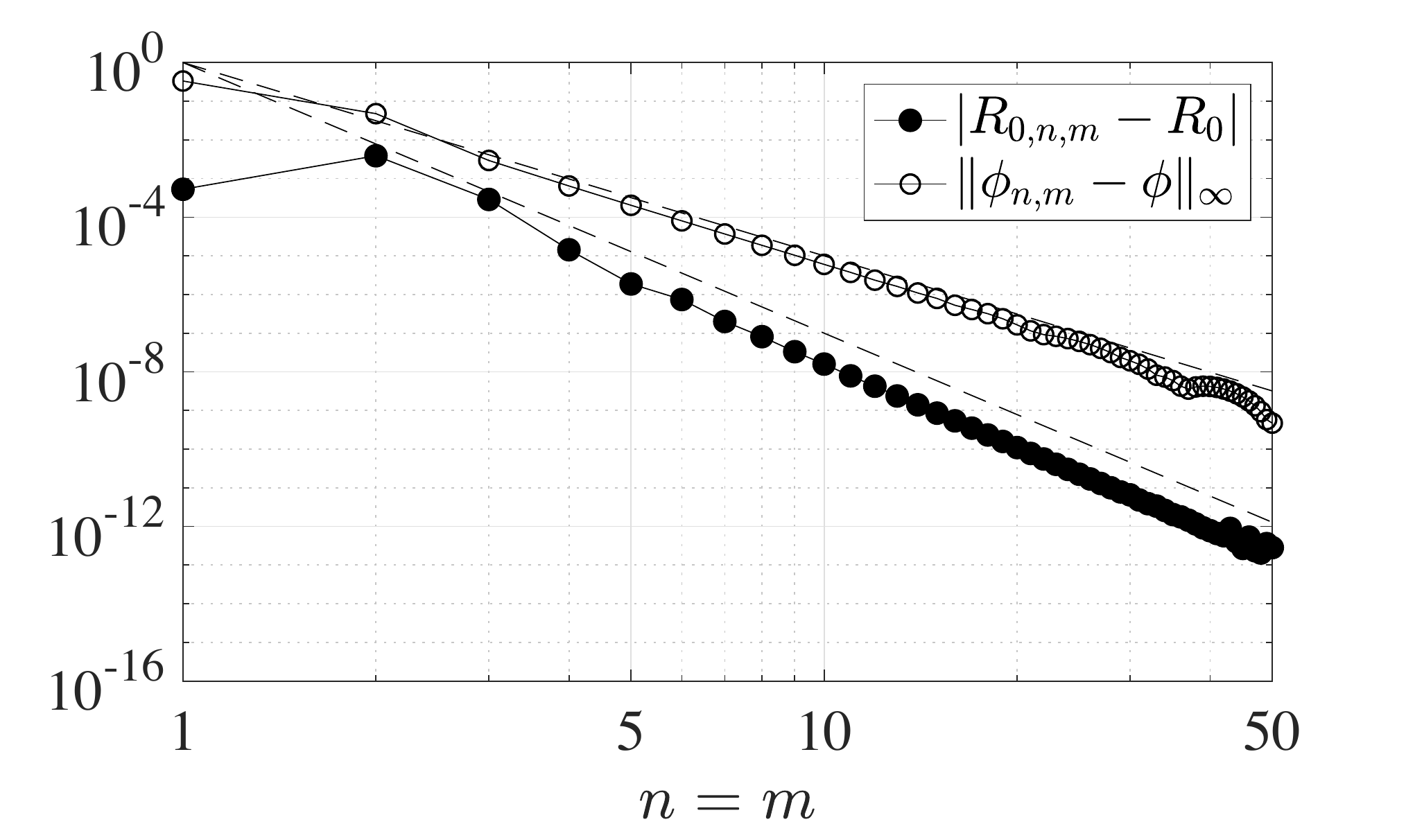}
\caption{Errors relevant to Example 1 (left) and Example 2 (right, the dashed lines indicate order $5$ and $7$), see Table \ref{t1} and the text for more details.}
\label{f_ex12}
\end{center}
\end{figure}

\paragraph{Example 2}For the choices listed in the third column of Table \ref{t1}, \eqref{geig} gives
\begin{equation*}
R_{0}=6/77\approx0.077922077922078,\qquad\phi(x,y)=x^{5/2}y^{8/3}.
\end{equation*}
The trend of the errors for increasing $n=m$ on both $R_{0}$ and $\phi$ are reported in Figure \ref{f_ex12} (right). Convergence of order 5 is observed for the error on $\phi$, order 7 for the error on $R_{0}$. We believe that this could come from the specific forms of $B$ and $M$, indeed we have
\begin{equation*}
(B\phi)(x,y)=\phi(x,y)\int_{x_{0}}^{\bar x}\int_{y_{0}}^{\bar y}\phi(\xi,\sigma)\dd\sigma\dd\xi
\end{equation*}
and $M\phi=\phi$, which give the eigenfunction $\phi$ and
\begin{equation*}
R_{0}=\int_{x_{0}}^{\bar x}\int_{y_{0}}^{\bar y}\phi(\xi,\sigma)\dd\sigma\dd\xi.
\end{equation*}
Of course a full explanation requires a rigorous error analysis as, e.g., the one developed in \cite{bkrv20} for the case of single structure. As already remarked, such analysis is out of the scope of the current investigation.

\paragraph{Example 3}For the choices listed in the fourth column of Table \ref{t1}, \eqref{geig} gives
\begin{equation*}
R_{0}=\frac{1}{C}\approx3.178501217245177,\qquad\phi(x,y)=e^{-x}y^{7/2}.
\end{equation*}
The trend of the errors for increasing $n=m$ on both $R_{0}$ and $\phi$ are reported in Figure \ref{f_ex3}. Convergence of order 7 is observed for the error on $\phi$, order 9 for the error on $R_{0}$. As the forms of $B$ and $M$ are the same of those in Example 2, this couple of examples experimentally confirm that the rate of convergence is proportional to the degree of smoothness of the eigenfunction.
\begin{figure}
\begin{center}
\includegraphics[width=.49\textwidth]{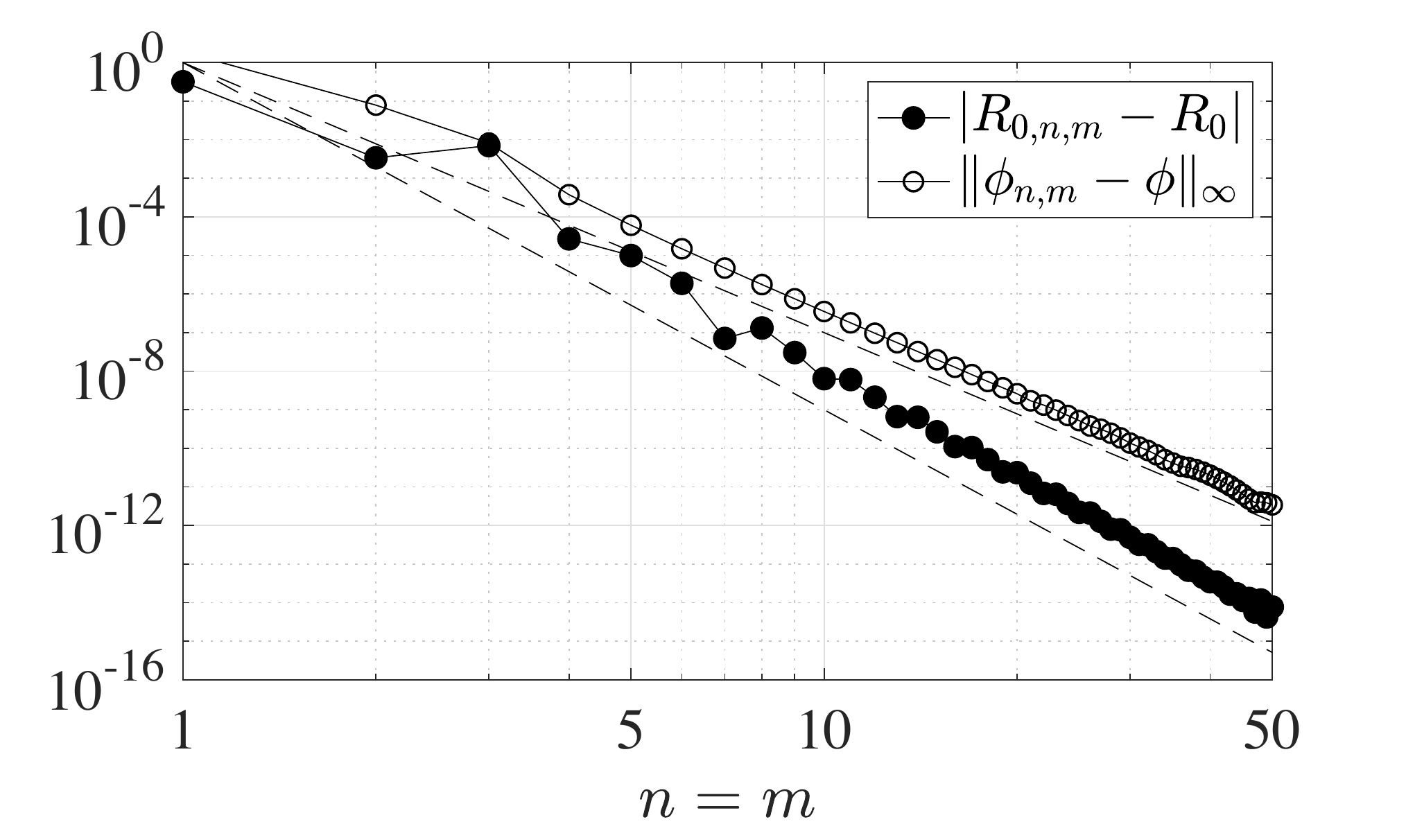}
\caption{Errors relevant to Example 3 (the dashed lines indicate order $7$ and $9$), see Table \ref{t1} and the text for more details.}
\label{f_ex3}
\end{center}
\end{figure}
\section{An age-immunity model}
\label{s_ageimmunity}
We first present in Section 5.1 an epidemic model structured by age and immunity, introduced with the goal to set a framework to investigate childhood diseases (e.g., pertussis) and their vaccination programs. After computing the disease-free equilibrium and linearizing around the latter in Section \ref{s_ai_dfe}, we derive explicit expressions for the operators $B$ and $M$ and prove the compactness of the NGO $BM^{-1}$ in Section \ref{s_ai_ngo}. Finally, in Section \ref{s_ai_tests} we apply the proposed method to approximate $R_{0}$ and the relevant eigenfunction $\phi$ in cases for which these quantities can be computed analytically, as well as in cases for which this is unattainable.
\subsection{The model}
\label{s_ai_model}
We consider the transmission of a disease in a closed population where individuals are characterized by two different structures, namely demographic age $a\in[0,\bar a]$ and immunity level $w\in[0,1]$. The former determines for instance the natural mortality of individuals, say $\mu(a)$. The latter determines the epidemiological properties of individuals: in particular, $w=0$ corresponds to fully susceptible individuals while $w=1$ corresponds to fully protected ones. Furthermore, in order to model immunity waning in time, we assume that, for a susceptible individual, $w$ varies according to
\begin{equation*}
w'(t)=-g(w(t))
\end{equation*}
for some positive differentiable\footnote{See, e.g., \cite{web08,dou11}.} function $g$, whereas we assume that $w$ does not change during the infectious period of an individual. 

The given population is then divided in two classes formed by susceptible and infected individuals. From the biological point of view, we suppose that, upon contact with an infectious individual, a susceptible individual can get infected with probability $\beta(w)$ or its immune system can be boosted at the maximal immunity level with probability $\alpha(w)$. To fix ideas, we assume that $\alpha(w)+\beta(w)=1$. After being infected, an individual becomes immediately infective with infectivity $\nu(w)$ depending on its immunity level $w$: for instance, individuals with higher immunity level, once infected, may develop milder symptoms thus being less infectious. We furthermore assume that infected individuals have a constant recovery rate $\gamma>0$ and, upon recovery, they acquire full immunity ($w=1$), which then wanes in time with rate $g$. Lastly, for the sake of simplicity we assume no additional mortality from the disease. We remark however that, despite the overall fatality rate of childhood diseases is typically small, the mortality of infants, unable to be vaccinated, is of major concern for many immunization programs. This information could be easily incorporated in the equations by adding an age-dependent disease-induced mortality term.

\bigskip
Let $s(t,a,w)$ and $i(t,a,w)$ be respectively the density of susceptible and infected individuals. The dynamics of the susceptibles is described by
\begin{equation*}
\left\{\setlength\arraycolsep{0.1em}\begin{array}{l}
\partial_{t} s(t,a,w)+\partial_{a}s(t,a,w)-\partial_{w}[g(w)s(t,a,w)]=-[\mu(a)+\lambda(t,w)+\eta(t,w)]s(t,a,w),\\[2mm]
\displaystyle g(1)s(t,a,1)=\gamma\int_{0}^{1}i(t,a,w)\dd w+\int_{0}^{1}\eta(t,w)s(t,a,w)\dd w,\\[3mm]
s(t,0,w)=\mathcal B(w),
\end{array}\right.
\end{equation*}
where $\mathcal{B}(w)$ is the population birth rate, and
\begin{align*}
&\lambda(t,w):=\beta(w)\int_{0}^{1}\nu(\omega)\int_{0}^{\bar a}i(t,a,\omega)\dd a\dd\omega,\\
&\eta(t,w):=\alpha(w)\int_{0}^{1}\nu(\omega)\int_{0}^{\bar a}i(t,a,\omega)\dd a\dd\omega
\end{align*}
are, respectively, the force of infection and, using a similar terminology, the ``force of boosting'' acting on a susceptible individual with immunity level $w$. The dynamics of the infected individuals is described by

\begin{equation*}
\left\{\setlength\arraycolsep{0.1em}\begin{array}{l}
\partial_{t}i(t,a,w)+\partial_{a}i(t,a,w)=\lambda(t,w)s(t,a,w)-\mu(a)i(t, w, a)-\gamma i(t,a,w),\\[1mm]
i(t,a,1)=0,\\[1mm]
i(t,0,w)=0,
\end{array}\right.
\end{equation*}
where we assume absence of vertical transmission. Note that the equation for the infectives has no derivative with respect to $w$, as we assumed that $w$ does not change during the infectious period.
\subsection{Disease-free equilibrium and linearization}
\label{s_ai_dfe}
Let us look for a stationary solution $s(t,a,w)\equiv\bar s(a,w)$ in the absence of infected individuals, i.e., $i(t,a,w)\equiv\bar i(a,w)\equiv0$. As the force of infection and the force of boosting vanish, we are left with
\begin{equation}\label{PDEsbar}
\left\{\setlength\arraycolsep{0.1em}\begin{array}{l}
\partial_{a}\bar s(a,w)-\partial_{w}[g(w)\bar s(a,w)]=-\mu(a)\bar s(t,a,w),\\[1mm]
g(1)\bar s(a,1)=0,\\[1mm]
\bar s(0,w)=\mathcal{B}(w),
\end{array}\right.
\end{equation}
which can be solved through the method of characteristics as follows. The characteristic curves 
$a\rightarrow w(a)$ in the rectangle $[0,\bar a]\times[0,1]$ are the solutions of the IVP
\begin{equation}\label{IVPw}
\begin{cases}
w'(a)=-g(w(a)),\\
w(a_{0})=w_{0},
\end{cases}
\end{equation}
which is well defined for all $(a_{0}, w_{0})\in[0,\bar a]\times[0,1]$ thanks to the differentiability of $g$. In particular, $g(w)>0$ implies decreasing characteristic curves, so that they start either from $w(0)=w_{0}$ for some $w_{0}\in[0,1]$ or from $w(a_{0})=1$ for some $a_{0}\in[0,\bar a]$. Once that one of such curves is selected, we define
\begin{equation*}
\sigma(a):=\bar s(a,w(a)),
\end{equation*}
and it is not difficult to see from \eqref{PDEsbar} that
\begin{equation}\label{IVPsigma}
\sigma'(a)=[g'(w(a))-\mu(a)]\sigma(a)
\end{equation}
holds along the curve. The latter is a linear yet scalar nonautonomous ODE, and the related IVP can be easily solved, assuming that $\mu$ and $g$ are smooth enough to ensure existence and uniqueness. The initial condition is either $\sigma(0)=\bar s(0,w_{0})=\mathcal{B}(w_{0})$ for the solution along the characteristic curve starting at $(0,w_{0})$ for any $w_{0}\in[0,1]$ or $\sigma(a_{0})=\bar s(a_{0},1)=0$ for the solution along the characteristic curve starting at $(a_{0},1)$ for any $a_{0}\in[0,\bar a]$. A couple of examples are given below.

\paragraph{Example 4}Let us choose $g(w)=\gamma(c-w)$ for $c\geq1$, $\mathcal{B}(w)=(1-w)^2$ and $\mu(a)=\bar\mu>0$. By solving \eqref{IVPw} we obtain the characteristic curves starting at $(0,w_{0})$ as
\begin{equation*}
w(a)=c+e^{\gamma a}(w_{0}-c)
\end{equation*}
and those starting at $(a_{0},1)$ as
\begin{equation*}
w(a)=c+e^{\gamma (a-a_{0})}(1-c).
\end{equation*}
They both coincide with $w^{\ast}(a):=c+e^{\gamma a}(1-c)$ when $(a_{0},w_{0})=(0,1)$ (note that $w^{\ast}(a)\equiv1$ when $c=1$). Moreover, since \eqref{IVPsigma} reduces to
\begin{equation*}
\sigma'(a)=-\left(\gamma+\bar\mu\right)\sigma(a),
\end{equation*}
we get the disease-free equilibrium
\begin{equation*}
\bar s(a,w)=\begin{cases}
\mathcal{B}(c+e^{-\gamma a}(w-c))e^{-(\gamma+\bar\mu)a}&\text{ for }w\leq w^{\ast}(a),\\
0&\text{ for }w\geq w^{\ast}(a).
\end{cases}
\end{equation*}
See Figure \ref{f3456} for a couple of instances, viz. $c>1$ (top) and $c=1$ (bottom).
\begin{figure}[htbp]
\begin{center}
\includegraphics[width=.49\textwidth]{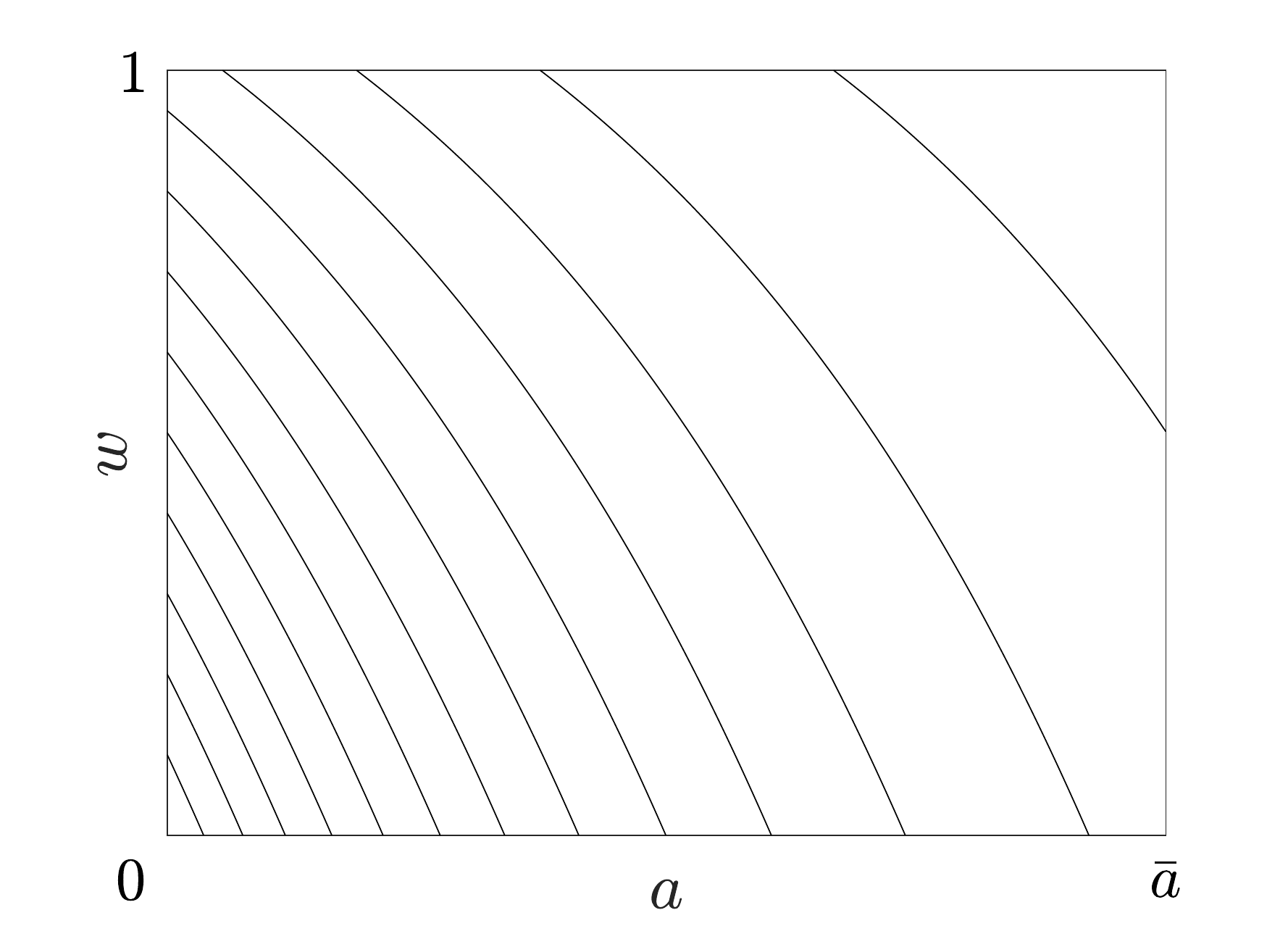}
\includegraphics[width=.49\textwidth]{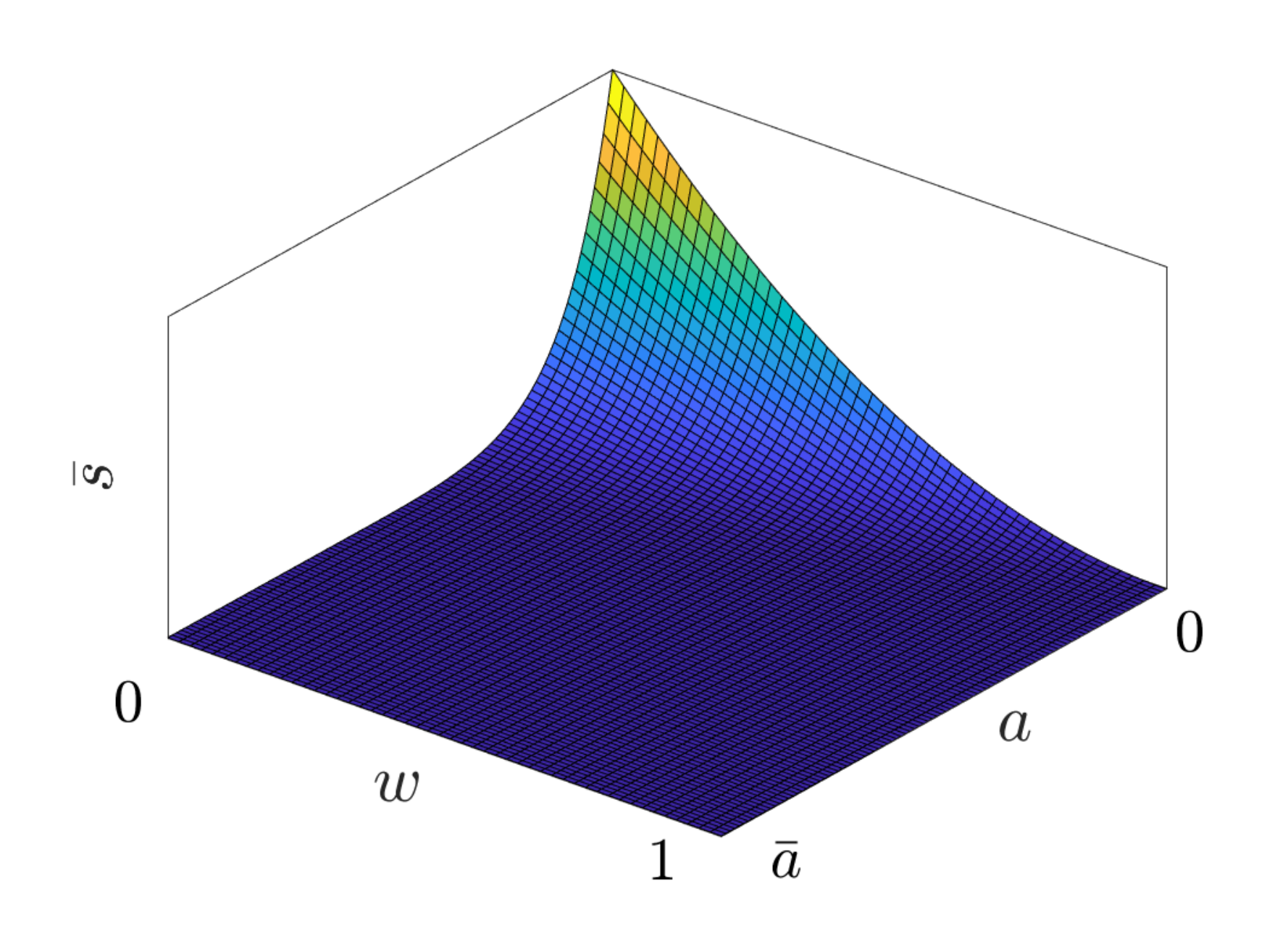}
\includegraphics[width=.49\textwidth]{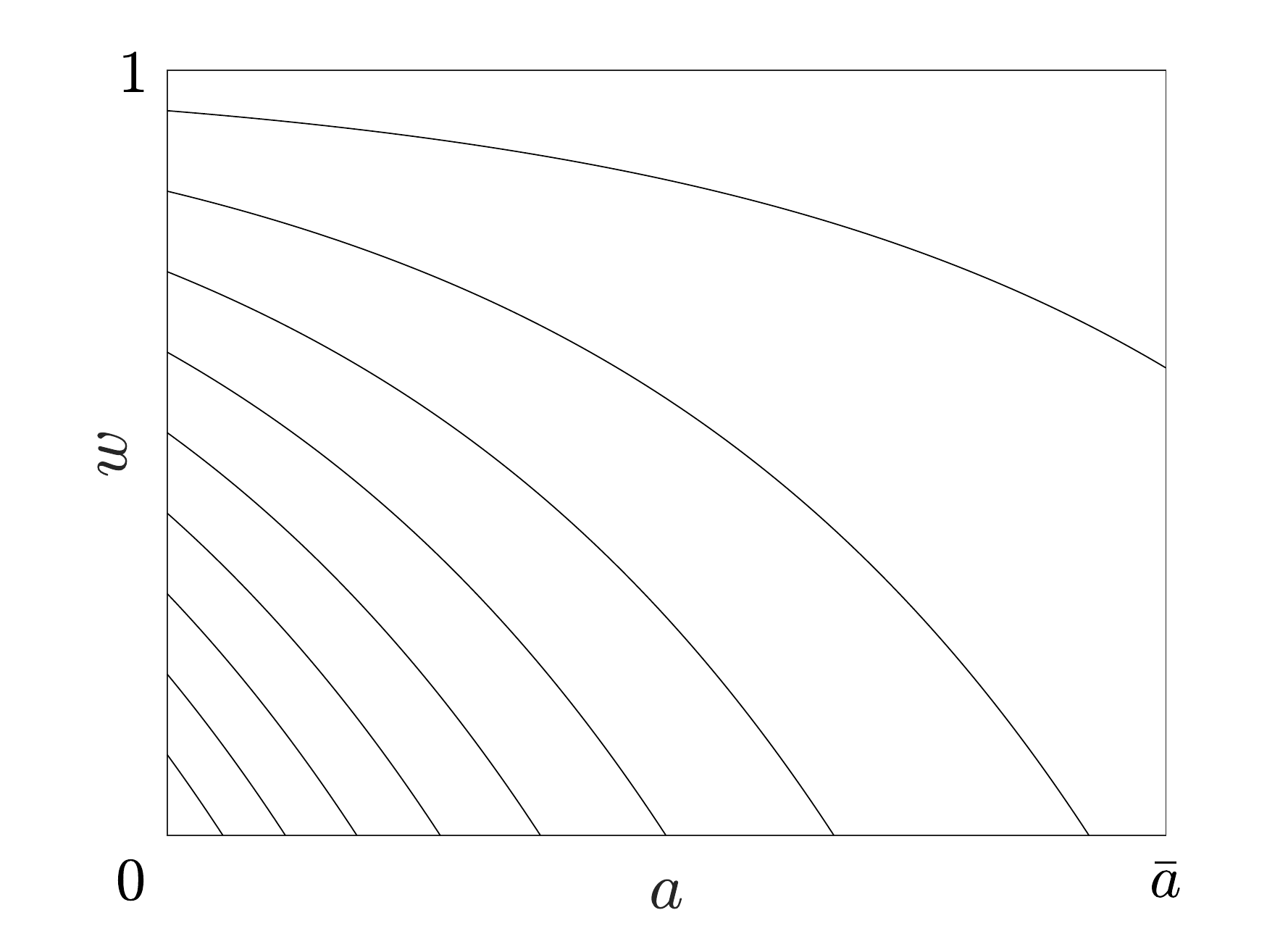}
\includegraphics[width=.49\textwidth]{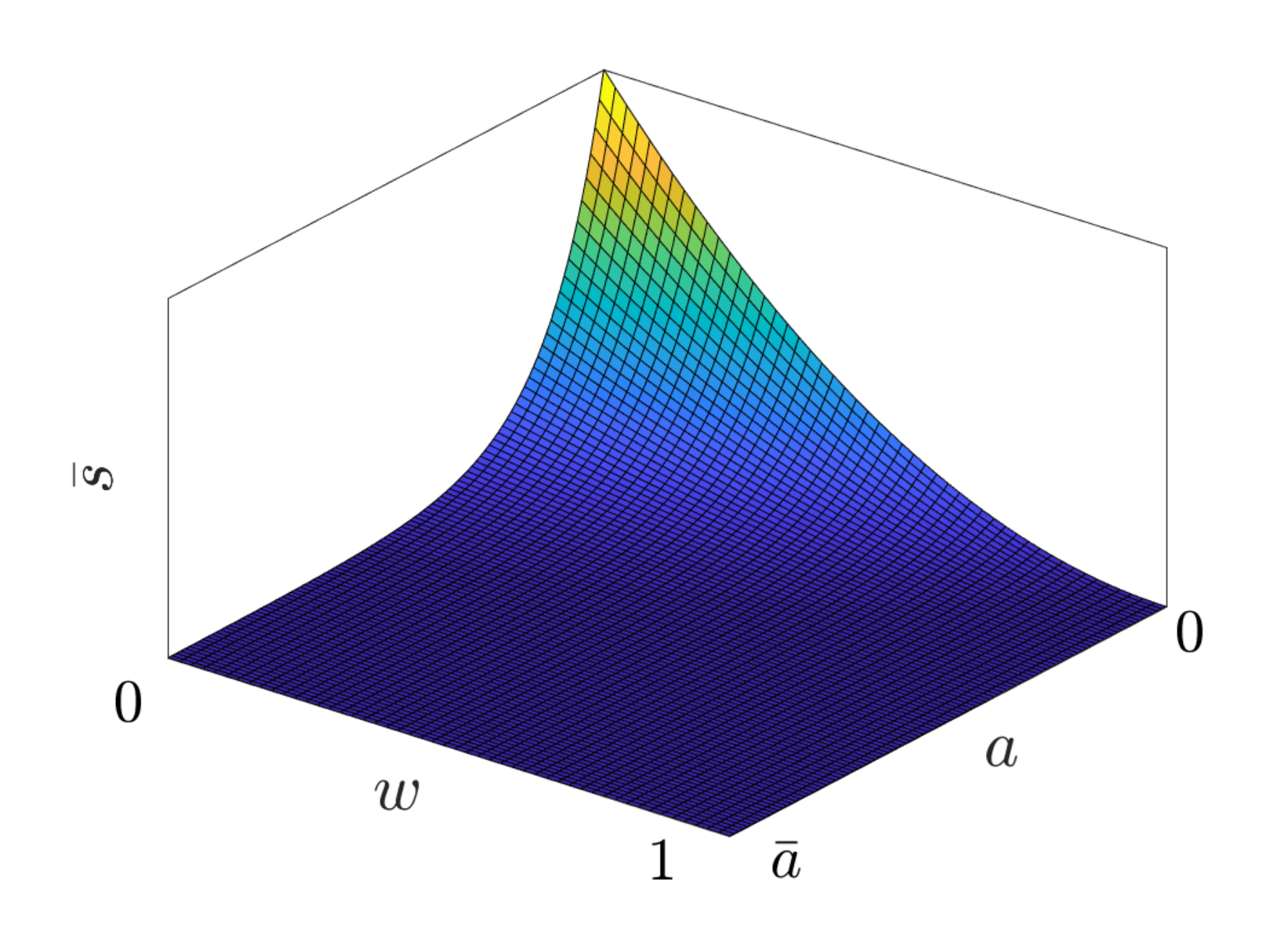}
\caption{Characteristic curves of \eqref{PDEsbar} (left) and susceptibles at the disease-free equilibrium (right) for Example 4 with $\bar a=2$, $\bar\mu=\gamma=1$ and $c=1.5$ (top) or $c=1$ (bottom).}\label{f3456}
\end{center}
\end{figure}

\paragraph{Example 5}Let us choose $g(w)=w$, $\mathcal{B}(w)=(1-w)^2$ and $\mu(a)=1/(\bar a-a)^2$. By solving \eqref{IVPw} we obtain the characteristic curves starting at $(0,w_{0})$ as
\begin{equation*}
w(a)=e^{-a}w_{0}
\end{equation*}
and those starting at $(a_{0},1)$ as
\begin{equation*}
w(a)=e^{-(a-a_{0})}.
\end{equation*}
They both coincide with $w^{\ast}(a):=e^{-a}$ when $(a_{0},w_{0})=(0,1)$. Moreover, since \eqref{IVPsigma} reduces to
\begin{equation*}
\sigma'(a)=\left(1-\frac{1}{(\bar a- a)^2}\right)\sigma(a),
\end{equation*}
we get the disease-free equilibrium
\begin{equation}\label{sbar2}
\bar s(a,w)=\begin{cases}
\mathcal{B}(we^{a})e^{a}e^{1/\bar a}e^{-1/(\bar a-a)}&\text{ for }w\leq w^{\ast}(a),\\
0&\text{ for }w\geq w^{\ast}(a).
\end{cases}
\end{equation}
See Figure \ref{f12} for a specific instance.
\begin{figure}[htbp]
\begin{center}
\includegraphics[width=.49\textwidth]{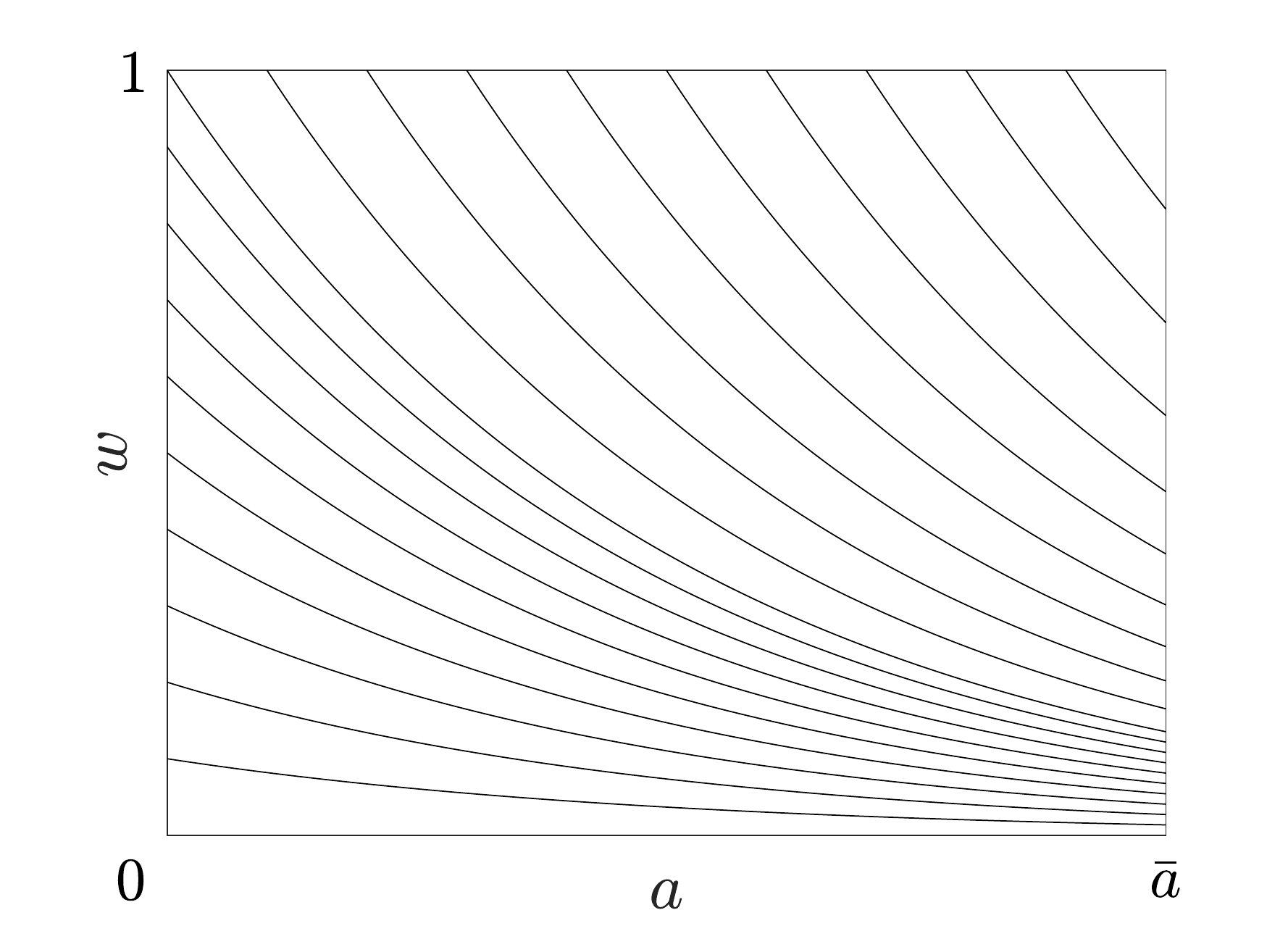}
\includegraphics[width=.49\textwidth]{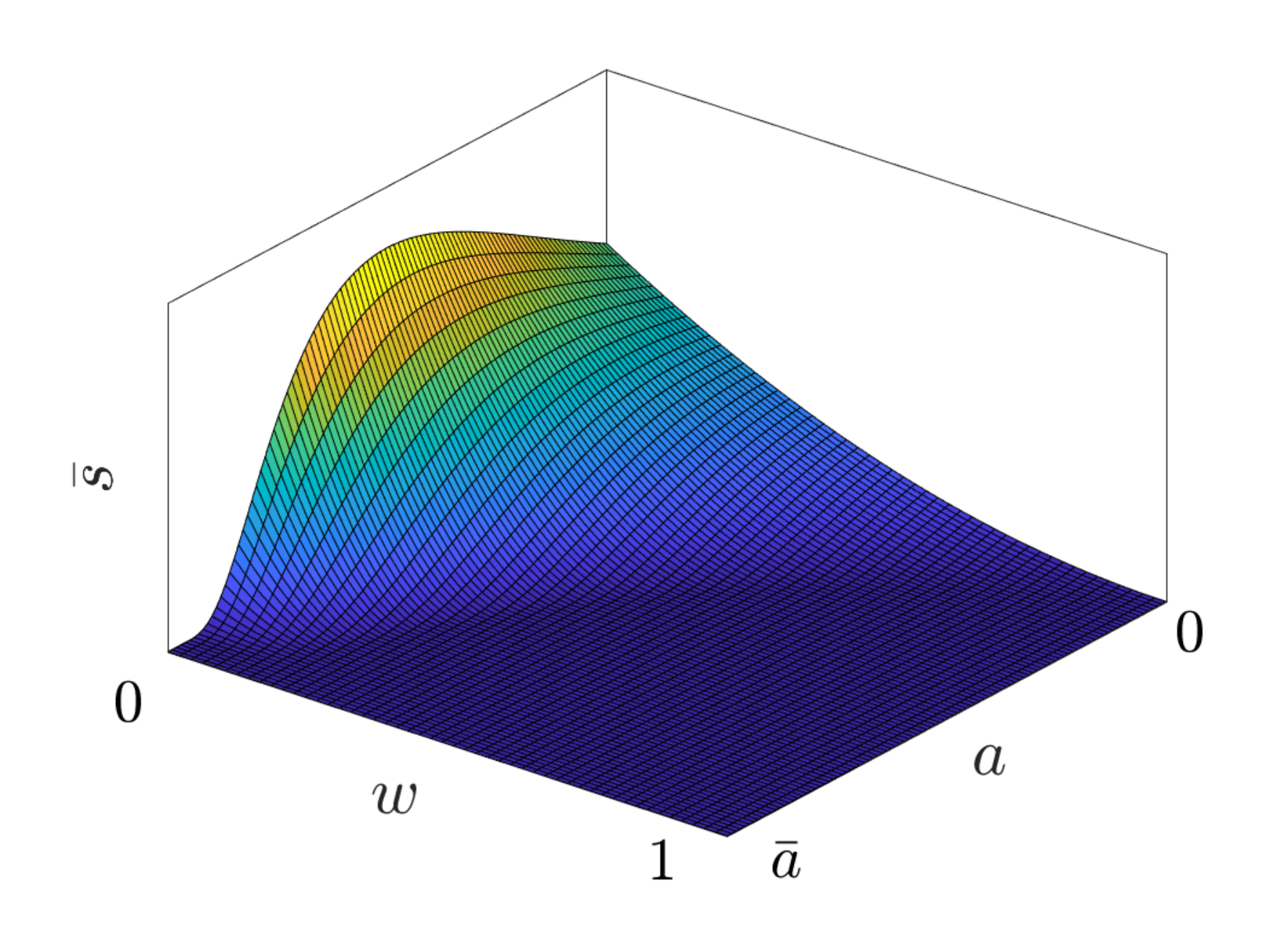}
\caption{Characteristic curves of \eqref{PDEsbar} (left) and susceptibles at the disease-free equilibrium (right) for Example 5 with $\bar a=2$.}\label{f12}
\end{center}
\end{figure}

\bigskip
Now, by defining the local perturbations
\begin{equation*}
u(t,a,w):=s(t,a,w)-\bar s(a,w),\qquad v(t,a,w):=i(t,a,w)
\end{equation*}
with respect to the disease-free equilibrium (recall that $\bar i(a,w)\equiv0$), we get the two linearized PDEs
\begin{equation*}
\left\{\setlength\arraycolsep{0.1em}\begin{array}{l}
\partial_{t}u(t,a,w)+\partial_{a}u(t,a,w)-\partial_{w}(g(w)u(t,a,w))=-\mu(a)u(t,a,w)-[\tilde\lambda(t,w)+\tilde\eta(t,w)]\bar s(a,w),\\[1mm]
\displaystyle g(1)u(t,a,1)=\gamma\int_{0}^{1} v(t,a,w)\dd w+\int_{0}^{1}\tilde\eta(t,w)\bar s(a,w)\dd w,\\[1mm]
u(t,0,w)=0,
\end{array}\right.
\end{equation*}
and
\begin{equation}\label{PDEy}
\left\{\setlength\arraycolsep{0.1em}\begin{array}{l}
\partial_{t}v(t,a,w)+\partial_{a}v(t,a,w)=\tilde\lambda(t,w)\bar s(a,w)-[\mu(a)+\gamma]v(t,a,w),\\[1mm]
v(t,a,1)=0,\\[1mm]
v(t,0,w)=0,
\end{array}\right.
\end{equation}
where
\begin{align*}
\tilde\lambda(t,w)&:=\beta(w)\int_{0}^{1}\nu(\omega)\int_0^{\bar a} v(t,a,\omega)\dd a\dd\omega, \\
\tilde\eta(t,w)&:=\alpha(w)\int_{0}^{1}\nu(\omega)\int_0^{\bar a} v(t,a,\omega)\dd a\dd\omega.
\end{align*}
Note that the PDE for the infectives $v$ is independent of the susceptibles $u$.
\subsection{Next generation operator, compactness and $R_{0}$}
\label{s_ai_ngo}
Set now $X:=L^{1}\left([0,\bar a]\times[0,1],\mathbb{R}\right)$. From \eqref{PDEy} we can define
\begin{equation*}
(B\phi)(a,w):=\beta(w)\left(\int_{0}^{1}\nu(\omega)\int_{0}^{\bar a}\phi(\xi, \omega)\dd \xi\dd\omega\right)\bar s(a,w)
\end{equation*}
and
\begin{equation*}
(M\phi)(a,w):=\partial_{a}\phi(a,w)+[\mu(a)+\gamma]\phi(a,w),
\end{equation*}
with
\begin{equation*}
D(M)=\left\{\phi\in X\ :\ \partial_{a}\phi\in X\text{ and }\phi(0,w)=\phi(a,1)=0\right\}.
\end{equation*}
To obtain an explicit expression for the NGO, we begin by observing that $M$ can be inverted by solving $M\phi=\psi$ for a given $\psi\in X$, i.e., the IVP 
\begin{equation*}
\left\{\setlength\arraycolsep{0.1em}\begin{array}{l}
\partial_{a}\phi(a,w)=-[\mu(a)+\gamma]\phi(a,w)+\psi(a,w)\\[1mm]
\phi(0,w)=0.
\end{array}\right.
\end{equation*}
Then we get
\begin{equation*}
(M^{-1}\psi)(a,w)=\int_{0}^{a} T(a,\xi)\psi(\xi,w)\dd\xi
\end{equation*}
for
\begin{equation*}
T(a,\xi):=e^{\displaystyle-\int_{\xi}^{a}[\mu(\sigma)+\gamma]\dd\sigma}.
\end{equation*}
We can thus define the NGO as
\begin{equation*}
(BM^{-1}\psi)(a,w):=\beta(w)\bar s(a,w)\left(\int_{0}^{1}\nu(\omega)\int_{0}^{\bar a}(M^{-1}\psi)(\xi, \omega)\dd \xi\dd\omega\right).
\end{equation*}
It is evident that the NGO is a positive operator. Next we show that it is also compact under mild assumptions. We first recall the following result from \cite{bre11} (where $\tau_{h}f:=f(\cdot+h)$).
\begin{lemma}\label{lemmaconv}
Let $G\in L^{q}(\mathbb{R}^{n})$ with $1\leq q<+\infty$. Then $\lim_{\|h\|_{\ast}\to0}\| \tau_{h}G-G\|_{q}=0$ for $\|\cdot\|_{\ast}$ any norm in $\mathbb{R}^{n}$.\end{lemma}
\begin{theorem} If $\beta,\nu\in L^{\infty}([0,1],\mathbb{R})$, $\mu\in L^{1}_{+}([0,\bar a],\mathbb{R})$ and $\bar s\in X$ then the NGO is compact.
\end{theorem}
\begin{proof}
Set $K:=BM^{-1}$ for brevity. We begin by observing that being $\mu\in X_{+}$ we have that $T(a,b)$ is continuous in $[0,\bar a]^2$ and hence bounded, say by a constant $C>0$. Furthermore we observe that under the above assumptions $K\psi\in X$ for all $\psi\in X$ and that, being $[0,\bar a]\times[0,1]$ a set of finite Lebesgue-measure, we have that $\beta\bar s\in X$. In view of applying the Kolmogorov-Riesz-Fr\'echet Theorem to prove compactness, we extend all the functions by zero outside $[0,\bar a]\times[0,1]$. Then we fix $m>0$, consider the set $U:=\{\psi\in X\ :\ \|\psi\|_{X}\leq m\}$ and prove that
\begin{equation}\label{cosadavedere}
\lim_{(h_{1},h_{2})\to(0,0)}\int_{\mathbb{R}}\int_{\mathbb{R}}\left|(K\psi)(a+h_{1},w+h_{2})-(K\psi)(a,w)\right|\dd w\dd a=0
\end{equation}
uniformly with respect to $\psi\in U$. By defining 
\begin{equation}\label{k}
k(\psi):=\int_{0}^{1}\nu(\omega)\int_o^{\bar a}(M^{-1}\psi)(\xi, \omega)\dd \xi\dd\omega,\quad\psi\in X,
\end{equation}
we have
\begin{equation*}
\setlength\arraycolsep{0.1em}\begin{array}{rcl}
|k(\psi)|&\leq&\displaystyle\int_{0}^{1}\int_{0}^{\bar a}\int_{0}^{\xi}\big|T(\xi, b)\big|\big|\psi(b, \omega)\big|\big|\nu(\omega)\big|\dd b\dd \xi\dd\omega\\[3mm]
&\leq&\displaystyle C\bar a\int_{0}^{\bar a}\int_{0}^{1}\big|\psi(b, \omega)\big|\big|\nu(\omega)\big|\dd\omega\dd b\\[3mm]
&\leq&\displaystyle C\bar a\|\nu\|_{L^\infty([0,1],\mathbb{R})}\|\psi\|_{X}.
\end{array}
\end{equation*}
Then, by defining $H:=C\bar a\|\nu\|_{L^\infty([0,1],\mathbb{R})}$ we arrive at
\begin{equation*}
\setlength\arraycolsep{0.1em}\begin{array}{rcl}
\displaystyle\int_{\mathbb{R}}\int_{\mathbb{R}}\big |(K\psi)(a+h_{1},w+h_{2})&-&(K\psi)(a,w)\big|\dd w\dd a\\[2mm]
&\leq&\displaystyle Hm\int_{\mathbb{R}}\int_{\mathbb{R}}\big|\beta(w+h_{2})\bar s(a+h_{1},w+h_{2})\\[3mm]
&&-\beta(w)\bar s(a,w)\big|\dd w\dd a,
\end{array}
\end{equation*}
which converges to $0$ uniformly with respect to $\psi$ as $(h_{1},h_{2})\to (0,0)$ thanks to Lemma \ref{lemmaconv}.
\end{proof}

Finally, we are able to get an explicit expression for $R_{0}$. To this aim, it is better to resort to \eqref{eig}, which becomes
\begin{equation*}
\lambda\psi(a,w)=\beta(w)\bar s(a,w)\left(\int_{0}^{1}\nu(\omega)\int_{0}^{\bar a}\int_{0}^{\xi}T(\xi,b)\psi(b,\omega)\dd b\dd\xi\dd\omega\right).
\end{equation*}
Then observe that the term between parentheses corresponds to $k$ in \eqref{k}, which is a constant, yet depending (linearly) on $\psi$. Therefore, $\psi(a,w)=\beta(w)\bar s(a,w)$ is an eigenfunction (modulo normalization) with corresponding eigenvalue
\begin{equation}\label{kbsbar}
\lambda=k(\beta\bar s).
\end{equation}
In particular, this is the only eigenvalue and the compactness of the NGO ensures that $R_{0}=\lambda$. Eventually, we observe from \eqref{k} that $R_{0}$ is in fact positive if $\nu$, $\beta$ and $\bar s$ are positive.
\subsection{Computation of $R_{0}$}
\label{s_ai_tests}
\paragraph{Example 6} Let us consider again Example 4. By choosing $\bar a=2$, $c=\gamma=\bar\mu=1$ and $\beta(w)=\nu(w)=1-w$, it is not difficult to recover explicitly
\begin{equation*}
R_{0}=\frac{1}{20}\left[\frac{1}{2}e^{-8}-e^{-4}+\frac{1}{2}\right]
\approx 0.024092604621261
\end{equation*}
with corresponding eigenfunction 
\begin{equation*}
\phi(a,w)=\frac{1}{2}(1-w)^3e^{-2a}\left[1-e^{-2a}\right].
\end{equation*}
The trend of the errors for increasing $n=m$ on both $R_{0}$ and $\phi$ are reported in Figure \ref{f_ai} (left). Convergence of infinite order is observed as expected.
\begin{figure}
\begin{center}
\includegraphics[width=.49\textwidth]{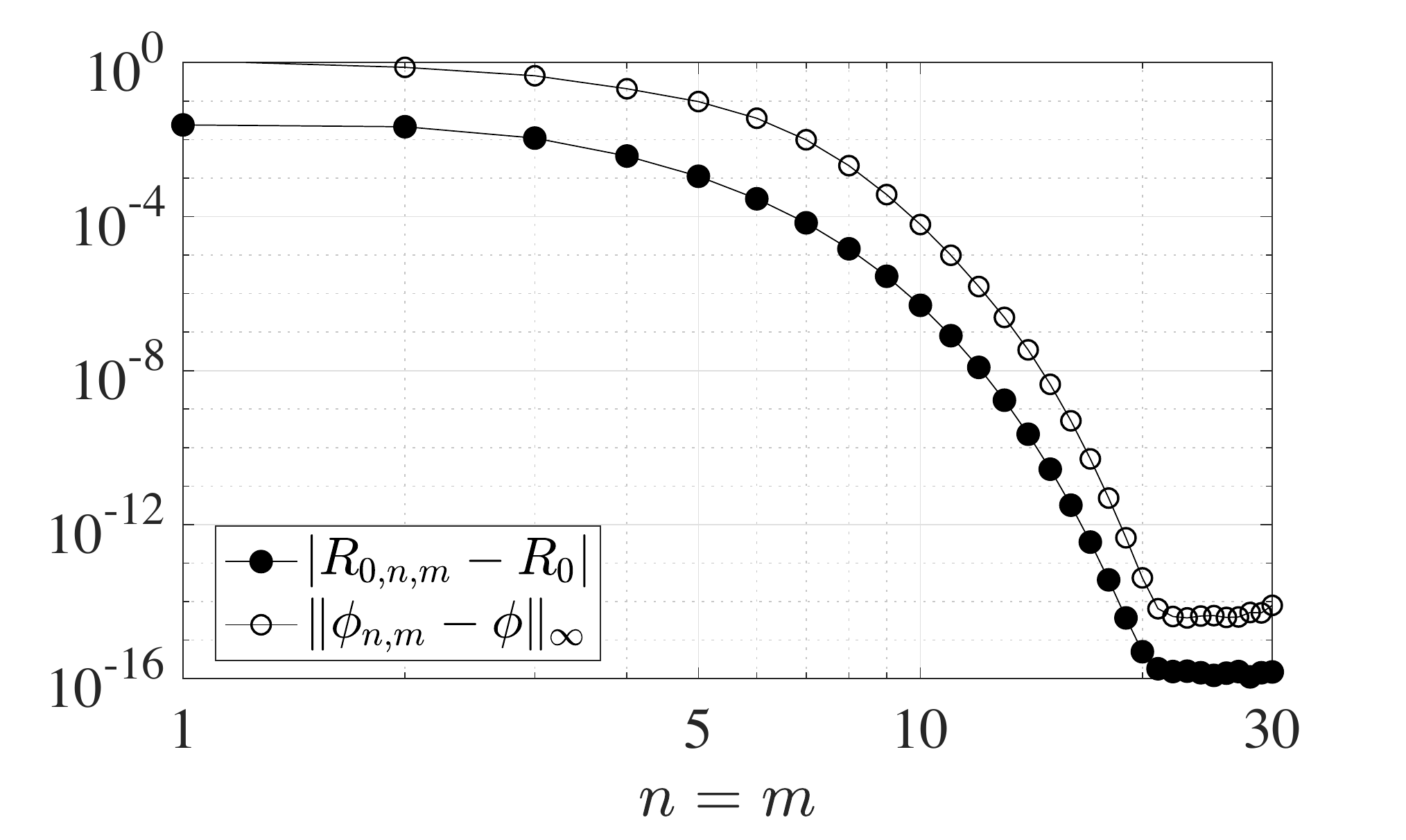}
\includegraphics[width=.49\textwidth]{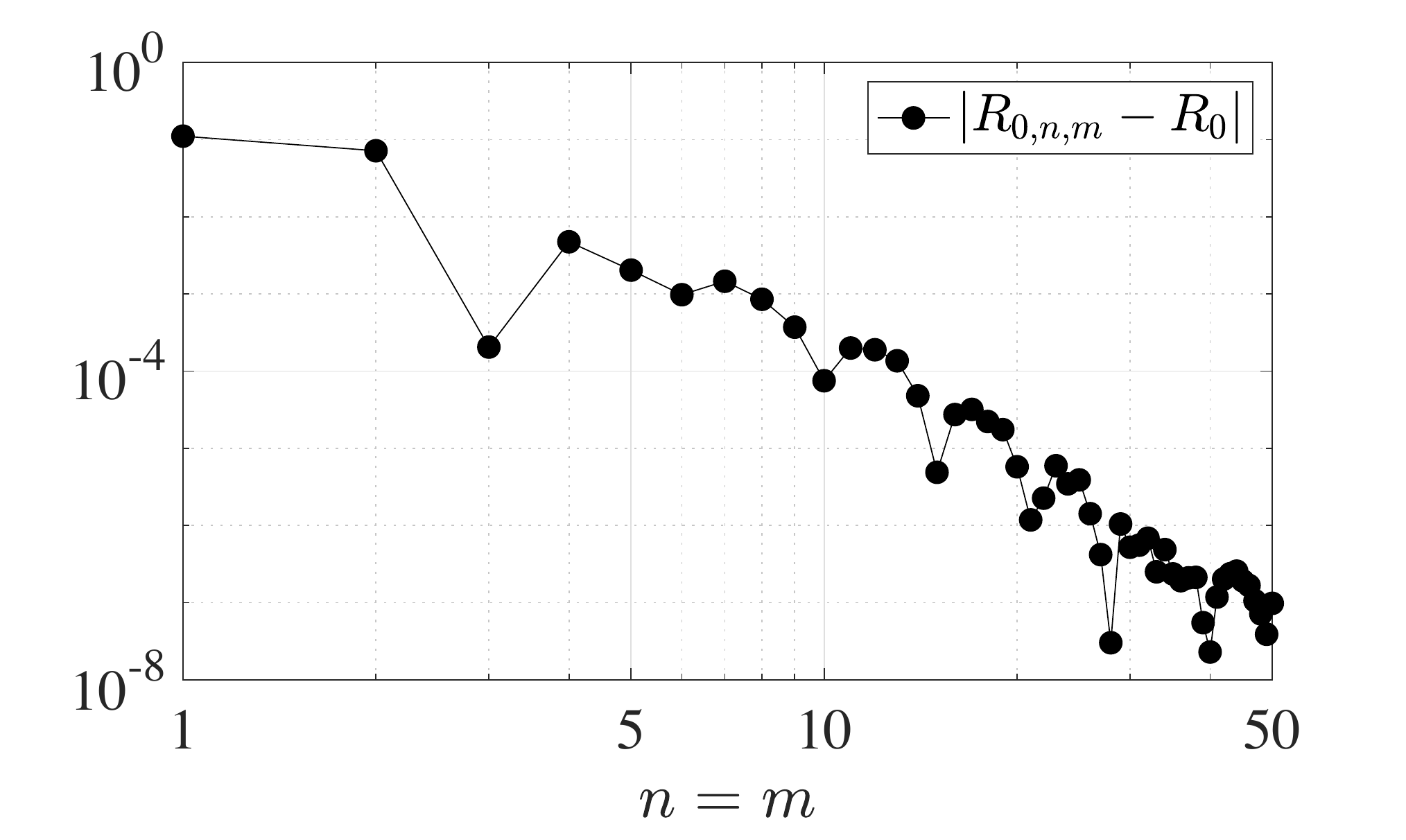}
\caption{Errors relevant to Example 6 (left) and Example 7 (right), see the text for more details.}
\label{f_ai}
\end{center}
\end{figure}

\paragraph{Example 7} Let us consider again Example 5, with $\bar a=2$, $\gamma=1$ and $\beta(w)=\nu(w)=1-w$. In this case it is not straightforward to compute explicitly $R_{0}$ (nor the corresponding eigenfunction). We thus assume as a reference value for $R_{0}$ the one computed with $n=m=100$, which reads
\begin{equation*}
R_{0}\approx 0.111258187908847.
\end{equation*}
The trend of the error for increasing $n=m$ on $R_{0}$ is reported in Figure \ref{f_ai} (right). Convergence of finite order occurs, which indicates lack of smoothness. In fact, by recalling \eqref{kbsbar} and \eqref{sbar2}, we observe that by taking the second derivative of $\bar s$ with respect to $w$ along any line with constant $a$ at $w=e^{-a}$ from above we are left with $\partial_{ww}\bar s(a,w)=2e^{3a}e^{1/\bar a}e^{-1/(\bar a-a)}\neq0$, while from below it obviously vanishes.
\section{Concluding remarks}
\label{s_concluding}
We have proposed a numerical method, based on bivariate collocation and cubature on tensor grids, to approximate $R_{0}$ in epidemic models with two structuring variables.
Essentially, the abstract linear equation describing the dynamics of the perturbations around the disease-free equilibrium is approximated with a finite-dimensional system, and the corresponding NGO is approximated with a matrix. 
$R_{0}$, which is the dominant eigenvalue of the NGO (when the latter is compact), is then approximated by the dominant eigenvalue of the approximating matrix.

We have provided several examples in which the corresponding eigenfunction is smooth and the experimental convergence to $R_{0}$ is of infinite order. For less regular eigenfunctions, the observed convergence is of finite order, as typically expected from polynomial-based approximations. The potential infinite order of convergence is fundamental to obtain good approximations with low-dimensional matrices.

Structured population models allow to incorporate more realistic features than, for instance, simpler ODE models. However, it is reasonable to say that the higher complexity and the lack of numerical tools to handle this type of systems hamper their use in real applications.
Due to the importance of $R_{0}$ in response to epidemic outbreaks, this work is of major relevance to provide modelers with numerical tools to analyze structured epidemic models, thus promoting their use in realistic applications.
In view of this, we have proposed a simple model structured by age and immunity that incorporates some of the essential features of childhood diseases, like waning and boosting of immunity. 
Having numerical methods at hand, this framework can become an effective support to study issues of public health relevance, like immunization programs.

Encouraged by the numerical investigations included in this paper, a fundamental next step will be to provide a rigorous proof of convergence, following the lines of \cite{bkrv20}. 
In this spirit, the numerical convergence observed here not only confirms that the order of convergence depends on the regularity of the eigenfunctions, but also supports the conjecture that no spurious eigenvalues of large modulus arise in the approximation. 

We plan to expand this work along several directions. 
One first question relates to multivariate polynomial approximation: here, we have considered models with two structures and we have used tensor products to construct the two-dimensional grid, the polynomials and the resulting differentiation matrices and cubature formulas. Having set the framework for models with two structures, it would be interesting to extend the technique to similar problems with more structures and provide a general computing framework. We do not expect theoretical complications from the numerical point of view (e.g., resorting to new approaches), despite the more involved discretization procedure. But a multivariate polynomial approximation of tensored form quickly becomes expensive. Alternative choices of discretization nodes constructed directly on the multi-dimensional domain may improve the computational cost required for cubature formulas and differentiation matrices. The Padua points in two dimensions \cite{bcdvx2,cmv08} and their generalizations to higher dimension \cite{bdv,bdv2} are well-suited to this aim.

Finally, we have here restricted to first-order hyperbolic PDEs, as they are the natural framework to treat variables that are characterized by a specific evolution in time. An interesting next step would be to consider models with more general ``spatial'' variables, including for instance diffusion processes. Also, in this paper we have restricted to the case of structuring variables belonging to bounded intervals. However, unbounded domains arise naturally in applications, for instance if variables are assumed to follow a normal or Gamma distribution. This case involves further numerical difficulties that could be addressed in future work. While in some situation it could be possible to apply a variable transformation to map the unbounded domain to a bounded one, an alternative approach could involve discretization and interpolation techniques specific to unbounded domains (e.g., \cite{gsv18,mn08}).
\section*{Acknowledgments}
\label{s_ackno}
DB, FS and RV are members of INdAM Research group GNCS and of UMI Research group ``Modellistica socio-epidemiologica''. The research of FS and JH was supported by the NSERC-Sanofi Industrial Research Chair in Vaccine Mathematics, Modeling and Manufacturing. FS is also supported by the UKRI through the JUNIPER modelling consortium [grant number MR/V038613/1].

\bigskip
\noindent
Declaration of interest: none


\end{document}